\documentclass{amsart}
\usepackage[foot]{amsaddr}
\usepackage{amssymb,amsmath,amsfonts,mathptmx,cite,mathrsfs,url}

\DeclareMathOperator{\alf}{alph}
\DeclareMathOperator{\mul}{mul}
\DeclareMathOperator{\simple}{sim}
\DeclareMathOperator{\var}{var}
\DeclareMathOperator{\ini}{ini}

\newtheorem{theorem}{Theorem}[section]
\newtheorem{proposition}[theorem]{Proposition}
\newtheorem{lemma}[theorem]{Lemma}
\newtheorem{corollary}[theorem]{Corollary}

\theoremstyle{remark}

\numberwithin{equation}{section}

\makeatletter

\renewcommand*\subjclass[2][2010]{\def\@subjclass{#2}\@ifundefined{subjclassname@#1}{\ClassWarning{\@classname}{Unknown edition (#1) of Mathematics Subject Classification; using '2010'.}}{\@xp\let\@xp\subjclassname\csname subjclassname@#1\endcsname}}

\renewcommand{\subjclassname}{\textup{2010} Mathematics Subject Classification}

\makeatother

\begin{document}

\title{Cross varieties of aperiodic monoids}
\thanks{S.V.Gusev was supported by the grant of the Russian Science Foundation No.~25-71-00005, \url{https://rscf.ru/en/project/25-71-00005/}.}

\author{Sergey V. Gusev}
\address{Institute of Natural Sciences and Mathematics, Ural Federal University, Lenina str.~51, 620000 Ekaterinburg, Russia}
\email{sergey.gusb@gmail.com}

\begin{abstract}
A variety of universal algebras is called Cross if it is finitely based, finitely generated, and has only finitely many subvarieties.
A monoid is aperiodic if all its subgroups are trivial.
In the present article, we show that a variety of aperiodic monoids is Cross if and only if it excludes, as subvarieties, a certain list of 22 almost Cross varieties.
Consequently, this list of 22 varieties exhausts all almost Cross varieties of aperiodic monoids.
\end{abstract}



\maketitle

\section{Introduction}
\label{introduction}

A \textit{variety} is a class of universal algebras of a fixed signature that is closed under the formation of homomorphic images, subalgebras, and arbitrary direct products.
According to Birkhoff's theorem~\cite{Birkhoff-35}, varieties are precisely the classes of universal algebras that satisfy a given set of identities.
A variety is said to be \textit{finitely based} if it can be defined by a finite set of identities, \textit{finitely generated} if it is generated by a finite algebra, and \textit{small} if its lattice of subvarieties is finite.
Following Higman~\cite{Higman-60}, a variety that satisfies all three conditions is called \textit{Cross}.

Finite groups~\cite{Oates-Powell-64}, associative rings~\cite{Lvov-73,Kruse-73}, Lie rings~\cite{Bahturin-Olshanski-75}, and lattices~\cite{McKenzie-70} generate Cross varieties.
Semigroups and monoids are, perhaps, the only ``classical'' types of universal algebras for which the analogous statement does not hold. In particular, there exist finite semigroups and monoids that generate non-finitely based varieties (see the survey~\cite{Volkov-01}). 
A finitely generated variety may also fail to be Cross due to the infiniteness of its lattice of subvarieties (see, for example, the recent survey~\cite{Araujo-Araujo-Cameron-Lee-Raminhos-23}).

For any class of algebras that contains non-Cross varieties, one approach to characterizing Cross varieties is to identify its minimal non-Cross varieties, which are commonly called \textit{almost Cross} varieties. 
It follows from Zorn's lemma that the exclusion of almost Cross subvarieties is not only necessary but also sufficient for a variety to be Cross.

The problem of finding almost Cross varieties of groups attracted considerable attention (see~\cite{Kovacs-Newman-71,Olshanski-71,Razmyslov-72}, for instance). 
In the general case, however, this problem has turned out to be transcendently difficult, in view of the following result~\cite{Kozhevnikov-12}: there exist continuum many non-finitely based almost Cross varieties of groups whose lattices of subvarieties are isomorphic to the three-element chain.
Nevertheless, no explicit example of such a variety is known to date.

The present article is concerned with Cross varieties of \textit{monoids}, i.e., semigroups with an identity element.
For a long time, only two explicit examples of non-group almost Cross varieties of monoids were known: the variety $\mathbf{Com}$ of all commutative monoids~\cite{Head-68} and the variety $\mathbf B$ of all idempotent monoids~\cite{Wismath-86}.
In 2005, Jackson~\cite{Jackson-05} discovered two further explicit examples.
Both of these examples consist of \textit{aperiodic monoids}, i.e., monoids all of whose subgroups are trivial.
Since then, the problem of classifying almost Cross varieties of aperiodic monoids has attracted much attention: a number of new examples of almost Cross varieties have been found, and descriptions of such varieties have also been obtained for several subclasses of the class of aperiodic monoids (see~\cite{Lee-13,Lee-14,Gusev-20,Gusev-25IJAC,Gusev-Sapir-22,Gusev-Li-Zhang-25,Gusev-Lee-Zhang-26,Gusev-Vernikov-18,Jackson-Lee-18,Sapir-21,Sapir-23}).
In the present paper, we complete these investigations and provide a complete description of almost Cross varieties of aperiodic monoids by giving an exhaustive finite list of such varieties.
Note that, in view of the above-mentioned result~\cite{Kozhevnikov-12}, one can hardly hope to describe almost Cross varieties of monoids in general.

This paper is structured as follows.
Some background information and preliminary results are presented in Section~\ref{sec: preliminaries}.
Then almost Cross varieties of aperiodic monoids and related information required to establish the main result are given in Section~\ref{sec: almost}.
In particular, three additional new almost Cross varieties (up to duality) are exhibited in Subsections~\ref{subsec: xyxz=xyxzx} and~\ref{subsec: A1}, bringing the total number of known almost Cross varieties of aperiodic monoids to 22.
In Section~\ref{sec: characterization}, it is shown that a variety of aperiodic monoids is Cross if and only if it excludes these 22 almost Cross varieties, thus completing the description of all almost Cross varieties of aperiodic monoids.

\section{Preliminaries}
\label{sec: preliminaries}

\subsection{Words, identities, Rees quotients of free monoids}

Let $\mathscr X^\ast$ denote the free monoid over a countably infinite alphabet~$\mathscr X$. 
Elements of~$\mathscr X$ are called \textit{letters} and elements of~$\mathscr X^\ast$ are called \textit{words}. 
We treat the identity element~1 of~$\mathscr X^\ast$ as \textit{the empty word}. 
The \textit{alphabet} of a word $\mathbf w$, i.e., the set of all letters occurring in $\mathbf w$, is denoted by $\alf(\mathbf w)$. 
For a word $\mathbf w$ and a letter $x$, let $|\mathbf w|_x$ denote the number of occurrences of $x$ in $\mathbf w$.
A letter $x$ is called \textit{simple} [\textit{multiple}] \textit{in a word} $\mathbf w$ if $x$ occurs in $\mathbf w$ once [more then once]. 
The set of all simple [multiple] letters in a word $\mathbf w$ is denoted by $\simple(\mathbf w)$ [respectively, $\mul(\mathbf w)$]. 

An \textit{identity} is an expression $\mathbf u\approx\mathbf v$, where $\mathbf u,\mathbf v\in\mathscr X^\ast$.
An identity $\mathbf u\approx\mathbf v$ is \textit{non-trivial} if $\mathbf u\ne\mathbf v$.
A monoid variety defined by a set $\Sigma$ of identities is denoted by $\var\,\Sigma$.
The subvariety of a variety $\mathbf V$ defined by a set $\Sigma$ of identities is denoted by $\mathbf V \Sigma$.
We say that a set $\Sigma_1$ of identities is \textit{equivalent modulo} $\Sigma$ to a set $\Sigma_2$ of identities if $\var(\Sigma\cup\Sigma_1)=\var(\Sigma\cup\Sigma_2)$.

For any set of words $\mathscr W\subseteq\mathscr X^\ast$, let $M(\mathscr W)$ denote the Rees quotient over the ideal of $\mathscr X^\ast$ consisting of all words that are not subwords of words in $\mathscr W$. 
Given any set $\mathscr W$ of words, let $\mathbf M(\mathscr W)$ denote the variety generated by the monoid $M(\mathscr W)$. 
A word~$\mathbf w$ is an \textit{isoterm} for a variety~$\mathbf V$ if it violates any non-trivial identity of the form $\mathbf w \approx\mathbf w^\prime$. 

\begin{lemma}[\!{\cite[Lemma~3.3]{Jackson-05}}] 
\label{L: isoterm}
For any variety~$\mathbf V$ and any set of words~$\mathscr W$, the inclusion $\mathbf M(\mathscr W)\subseteq\mathbf V$ holds if and only if all words from $\mathscr W$ are isoterms for~$\mathbf V$.\qed
\end{lemma}

For brevity, if $\mathbf w_1,\dots,\mathbf w_k\in\mathcal X^\ast$, then we will write $M(\mathbf w_1,\dots,\mathbf w_k)$ [respectively, $\mathbf M(\mathbf w_1,\dots,\mathbf w_k)$] rather than $M(\{\mathbf w_1,\dots,\mathbf w_k\})$ [respectively, $\mathbf M(\{\mathbf w_1,\dots,\mathbf w_k\})$].{\sloppy

}
\subsection{Varieties of commutative monoids and idempotent monoids}
\label{subsec: comm idem}

The variety $\mathbf{Com}$ of all commutative monoids is the first published example of a non-group almost Cross variety of monoids, while the variety $\mathbf{B}$ of all idempotent monoids is the first such example in the class of aperiodic monoids.{\sloppy

}
\begin{proposition}[\!{\cite{Head-68}}]
\label{P: Com}
The variety $\mathbf{Com}$ of all commutative monoids is a non-finitely generated almost Cross variety.
\end{proposition}

\begin{proposition}[\!{\cite[Proposition~4.7]{Wismath-86}}]
\label{P: B}
The variety $\mathbf B$ of all idempotent monoids is a non-finitely generated almost Cross variety.
\end{proposition}

The following claim is an immediate consequence of Propositions~\ref{P: Com} and~\ref{P: B}.

\begin{corollary}
\label{C: com id Cross}
\noindent
\begin{itemize}
\item[\textup{(i)}] Any commutative variety of aperiodic monoids is Cross.
\item[\textup{(ii)}] Any proper subvariety of $\mathbf B$ is Cross.
\end{itemize}
\end{corollary}

Let
\[
\mathbf B_2:=\var\{x\approx x^2,\,xy\approx xyx\},\quad
\mathbf B_3:=\var\{x\approx x^2,\,xyz\approx xyzxzyz\}.
\]
It is well known that $\mathbf B_2$ is the variety of left regular bands and is generated by the monoid obtained by adjoining an identity element to the left zero semigroup of order two.
The \textit{initial part} of a word $\mathbf w$, denoted by $\ini(\mathbf w)$, is the word obtained from $\mathbf w$ by retaining the first occurrence of each letter.
The following statement is well known and can be easily verified.

\begin{lemma}
\label{L: word problem LRB}
An identity $\mathbf u \approx \mathbf v$ holds in the variety $\mathbf B_2$ if and only if $\ini(\mathbf u)=\ini(\mathbf v)$.
\end{lemma}

Let $\ell(\mathbf u)$ denote the maximal prefix of $\mathbf u$ which contains all letters in $\mathbf u$ but one.
Dually, $r(\mathbf u)$ denotes the maximal suffix of $\mathbf u$ which contains all letters in $\mathbf u$ but one.

\begin{lemma}[\!{\cite[Lemma~2]{Green-Rees-52}}]
\label{L: x=xx}
An identity $\mathbf u \approx {\bf v}$ holds in $\mathbf B$ if and only if $\alf(\mathbf u) = \alf(\mathbf v)$ and the identities $\ell({\mathbf u}) \approx \ell(\mathbf v)$ and $r(\mathbf u) \approx r(\mathbf v)$ hold in $\mathbf B$.
\end{lemma}

\begin{lemma}[\!{\cite[Lemma~2.7]{Gusev-Sapir-22}}] 
\label{L: xy isoterm}
If $\mathbf V$ is a variety of aperiodic monoids such that $xy$ is not an isoterm for $\mathbf V$, then the variety $\mathbf V$ is either commutative or idempotent. 
\end{lemma}

\subsection{Small varieties satisfying $xyx\approx xyx^2$}

A variety is \textit{finitely generated} if it is generated by a finite algebra.
A variety is \textit{locally finite} if all of its finitely generated members are finite.

\begin{lemma}
\label{L: xyx=xyxx is finitely generated}
Each small monoid variety that satisfies the identity $xyx\approx xyx^2$ is finitely generated.
\end{lemma}

\begin{proof}
The required claim is a combination of the following two results: (1)~any variety satisfying the identity $xyx\approx xyx^2$ is locally finite~\cite{Sapir-87}; and (2)~any locally finite, small variety is finitely generated~\cite[Lemma~2.1]{Jackson-Lee-18}.
\end{proof}

\subsection{The varieties $\mathbf A_0$ and $\mathbf Q$}
\label{subsec: A01 Q1}

Let $\mathbf A_0$ and $\mathbf Q$ denote the varieties generated by the monoids
\[
\begin{aligned}
&A_0:=\langle a,b\mid a^2=a,\,b^2=b,\,ba=0\rangle\cup\{1\}=\{a,b,ab,0,1\},\\
&Q:=\langle a,b,c\mid a^2=a,\,ab=b,\,ca=c,\,ac=ba=cb=0\rangle\cup\{1\}=\{a,b,c,bc,0,1\},
\end{aligned}
\]
respectively.
It is shown in~\cite[Proposition 3.2(a)]{Edmunds-77} and~\cite[Chapter~4]{Lee-Li-11}, respectively, that 
\[
\begin{aligned}
&\mathbf A_0=\var\{xyzx\approx xyxzx,\,(xy)^2\approx (yx)^2\},\\
&\mathbf Q=\var\{xyx\approx xyx^2\approx x^2yx,\,x^2y^2\approx y^2x^2\}.
\end{aligned}
\]

\begin{lemma}
\label{L: A01 nsubseteq V}
Let $\mathbf V$ be a variety of monoids that satisfies the identities $xyx\approx xyx^2\approx x^2yx$. 
If $A_0\notin \mathbf V$, then $\mathbf V$ satisfies $x^2y^2\approx (xy)^2$. 
\end{lemma}

\begin{proof}
It follows from~\cite[Lemma~4.1]{Edmunds-77} that $\mathbf V$ satisfies an identity $x^py^q\approx \mathbf v$ such that $p,q\ge1$ and $\mathbf v$ contains $yx$ as a subword. 
Clearly, $x^{p+1}y^{q+1}\approx x\mathbf v(x,y)y$ is equivalent modulo $xyx\approx xyx^2\approx x^2yx$ to $x^2y^2\approx (xy)^2$.
\end{proof}

\begin{lemma}[\!{\cite[Lemma~2.12]{Gusev-25IJAC}}] 
\label{L: Q nsubseteq V}
Let $\mathbf V$ be a variety of monoids such that $M(xy)\in\mathbf V$ satisfying the identity $xyx\approx xyx^2$.
If $Q\notin \mathbf V$, then $\mathbf V$ satisfies $x^2yzx\approx x^2yxzx$.
\end{lemma}

\begin{lemma}[\!{\cite[Lemma~5.1]{Lee-14}}] 
\label{L: id Q1}
Let $\mathbf u := \mathbf u_0t_1\mathbf u_1\cdots t_m\mathbf u_m$ and $\mathbf v := \mathbf v_0h_1\mathbf v_1\cdots h_n\mathbf v_n$ be words with $\simple(\mathbf u)=\{t_1,\dots,t_m\}$ and $\simple(\mathbf v)=\{h_1,\dots,h_n\}$.
Then the monoid $Q$ satisfies the identity $\mathbf u\approx \mathbf v$ if and only if all of the following hold:
\begin{itemize}
\item[\textup{(a)}] $m = n$;
\item[\textup{(b)}] $t_i = h_i$ for all $i=1,\dots,n$;
\item[\textup{(c)}] $\alf(\mathbf u_i) = \alf(\mathbf v_i)$ for all $i=1,\dots,n$.
\end{itemize}
\end{lemma}

\subsection{The variety $\mathbf O$}
\label{subsec: O}

Let
\[
\mathbf O:=\var\{xzxyxty \approx xzyxty,\, xzytxyx \approx xzytyx\}.  
\]
We denote by $(\mathbf a_1,\mathbf a_2,\mathbf a_3,\mathbf a_4,\ldots)$ the alternating sequence $(x,y,x,y,\ldots)$.

\begin{proposition}[\!{\cite[Proposition~2.8]{Gusev-Lee-Zhang-26}}] 
\label{P: subvarieties of O cap J}
Every non-commutative subvariety $\mathbf V$ of the variety $\mathbf O\{(xy)^2\approx (yx)^2\}$ is defined within $\mathbf O\{(xy)^2\approx (yx)^2\}$ by $\Sigma_1 \cup \Sigma_2$ for some subset $\Sigma_1$ of
\[
\Phi_1:= \{xyx \approx yx^2,\, x^2yx \approx x^2y,\,  x^2yxzx \approx x^2yzx,\,  x^2yx\approx xyx\}
\] 
and some subset~$\Sigma_2$ of
\[
\Phi_2:=\left\{ \; \mathbf h\mathbf p\mathbf c \approx \mathbf h\mathbf q\mathbf c \; \; \left| 
\begin{array}{l} 
\mathbf h \in \{ 1,\, yh\}; \; \alf(\mathbf p)=\alf(\mathbf q)=\{x,y\}; \\[0.05in] 
\mathbf c \in \big\{ 1,\, txy,\, \prod_{i=1}^k (t_i\mathbf a_i), \,\prod_{i=2}^{k+1} (t_i\mathbf a_i) \,\big|\, k \geq 1 \big\}; \\[0.05in] 
1 \leq |\mathbf p|_x = |\mathbf q|_x, \, |\mathbf p|_y = |\mathbf q|_y \leq 2; \\[0.05in] 
|\mathbf h\mathbf p\mathbf c|_x = |\mathbf h\mathbf q\mathbf c|_x,\, |\mathbf h\mathbf p\mathbf c|_y = |\mathbf h\mathbf q\mathbf c|_y \geq 2 
\end{array} 
\right. \right\}.
\]
\end{proposition}

Analyzing the proof of Proposition~\ref{P: subvarieties of O cap J}, one can observe that the condition that $\mathbf{V}$ satisfies the identity $(xy)^2 \approx (yx)^2$ is not used anywhere in the case when the word $xy$ is an isoterm for $\mathbf{V}$.
Since every variety for which $xy$ is an isoterm is non-commutative, the following result is valid.

\begin{proposition} 
\label{P: subvarieties of O}
Every subvariety of $\mathbf O$ for which $xy$ is an isoterm is defined within $\mathbf O$ by $\Sigma_1 \cup \Sigma_2$ for some $\Sigma_1 \subseteq \Phi_1$ and $\Sigma_2 \subseteq \Phi_2$.
\end{proposition}

\begin{lemma}[\!{\cite[Corollary~2.9]{Gusev-Lee-Zhang-26}}]
\label{L: O cap J2}
Let $\mathbf V$ be a subvariety of $\mathbf O\{(xy)^2\approx (yx)^2\}$.
If, in addition, $\mathbf V$ satisfies the identity $x^2yzx\approx x^2yxzx$ or the identity 
\[
\sigma_m:\enskip  xy \biggl(\prod_{i=1}^m t_i\mathbf a_i\biggr) \approx yx \biggl(\prod_{i=1}^m t_i\mathbf a_i\biggr)
\]   
for some $m\ge 2$, then $\mathbf V$ is a Cross variety.
\end{lemma}

Similar to Proposition~\ref{P: subvarieties of O cap J}, one can observe that the condition that $\mathbf{V}$ satisfies the identity $(xy)^2 \approx (yx)^2$ is not used in the proof of Lemma~\ref{L: O cap J2} in the case when the word $xy$ is an isoterm for $\mathbf V$.
If the word $xy$ is not an isoterm for $\mathbf V$, then, since $\mathbf B$ violates $xzxyxty \approx xzyxty$ by Lemma~\ref{L: x=xx}, it follows from Corollary~\ref{C: com id Cross} and Lemma~\ref{L: xy isoterm} that $\mathbf V$ is a Cross variety.
Thus, the following result is valid.

\begin{lemma}
\label{L: O}
Let $\mathbf V$ be a subvariety of $\mathbf O$.
If, in addition, $\mathbf V$ satisfies the identity $x^2yzx\approx x^2yxzx$ or the identity $\sigma_m$ for some $m\ge 2$, then $\mathbf V$ is a Cross variety.
\end{lemma}

\begin{lemma}[\!{\cite[proof of the dual of Corollary~3.6]{Gusev-Sapir-22}}]
\label{L: implies O}
If $\mathbf V$ satisfies the identities $xyx\approx x^2yx$ and $xty^2x\approx xtxy^2x$, then $\mathbf V\subseteq\mathbf O$.
\end{lemma}

\begin{lemma}
\label{L: O-2}
Let $\mathbf V$ be a monoid variety.
If $\mathbf V$ satisfies the identities $xyx\approx x^2yx$, $xty^2x\approx xtxy^2x$ and
\[
\tau_m:\enskip  x^2y^2\biggl(\prod_{i=1}^m t_i\mathbf a_i\biggr)  \approx  xy^2x\biggl(\prod_{i=1}^m t_i\mathbf a_i\biggr)
\]
for some $m\ge 2$, then $\mathbf V$ is a Cross variety.
\end{lemma}

\begin{proof}
If the variety $\mathbf V$ is commutative or idempotent, then, since $\mathbf B$ violates $xty^2x\approx xtxy^2x$ by Lemma~\ref{L: x=xx}, it follows from Corollary~\ref{C: com id Cross} that $\mathbf V$ is a Cross variety.
Therefore, suppose that $\mathbf V$ is neither commutative nor idempotent, whence $M(xy)\in\mathbf V$ by Lemmas~\ref{L: isoterm} and~\ref{L: xy isoterm}.

According to Lemma~\ref{L: implies O}, $\mathbf V\subseteq\mathbf O$. 
Then, by Proposition~\ref{P: subvarieties of O}, $\mathbf V$ can be defined by $\Sigma:=\{xyx\approx x^2yx,\,xty^2x\approx xtxy^2x,\,\tau_m\}$ together with some identities in $\Phi_1\cup\Phi_2$. 
Consider an arbitrary identity $\mathbf h\mathbf p\mathbf c\approx \mathbf h\mathbf q\mathbf c$ in $\Phi_2$ with $\mathbf c\in \{t_1\mathbf a_1t_2\mathbf a_2\cdots t_n\mathbf a_n, \,t_2\mathbf a_2t_3\mathbf a_3\cdots t_{n+1}\mathbf a_{n+1}\}$ and $n > m+2$. Since $\mathbf V$ satisfies $xyx\approx x^2yx$, we may assume without loss of generality that $|\mathbf p|_x=|\mathbf p|_y=|\mathbf q|_x=|\mathbf q|_y=2$. 
Then $\mathbf p,\mathbf q\in\{x^2y^2,xy^2x,(xy)^2,y^2x^2,yx^2y,(yx)^2\}$. 
Notice also that
\begin{equation}
\label{xi_m+1}
x^2y^2 \biggl(\prod_{i=1}^{m+1} t_i\mathbf a_i\biggr)\stackrel{xyx\approx x^2yx}\approx x^2y^3 \biggl(\prod_{i=1}^{m+1} t_i\mathbf a_i\biggr) \stackrel{\tau_m}\approx xy^2xy \biggl(\prod_{i=1}^{m+1} t_i\mathbf a_i\biggr)\stackrel{xyx\approx x^2yx}\approx (xy)^2\biggl(\prod_{i=1}^{m+1} t_i\mathbf a_i\biggr).
\end{equation}
Hence if $\mathbf p,\mathbf q\in\{x^2y^2,xy^2x,(xy)^2\}$ or $\mathbf p,\mathbf q\in\{y^2x^2,yx^2y,(yx)^2\}$, then $\mathbf h\mathbf p\mathbf c\approx \mathbf h\mathbf q\mathbf c$ follows from $\{xyx\approx x^2yx,\,\tau_m\}$.

Now we assume without any loss of generality that $\mathbf p\in\{x^2y^2,xy^2x,(xy)^2\}$ and $\mathbf q\in\{y^2x^2,yx^2y,(yx)^2\}$.
In this case, $\Sigma\cup\{\mathbf h\mathbf p\mathbf c\approx \mathbf h\mathbf q\mathbf c\}$ implies the identity $\mathbf h\mathbf p\mathbf c^\prime\approx \mathbf h\mathbf q\mathbf c^\prime$ with $\mathbf c^\prime:=t_1\mathbf a_1t_2\mathbf a_2\cdots t_{m+2}\mathbf a_{m+2}$ because
\[
\begin{aligned}
\mathbf h\mathbf p\mathbf c^\prime\stackrel{\{\tau_m,\,\eqref{xi_m+1}\}}\approx{}& \mathbf h(xy)^2\mathbf c^\prime\stackrel{x^2\approx x^3}\approx \mathbf h(xy)^2(xy)^n\mathbf c^\prime \stackrel{\{\tau_m,\,\eqref{xi_m+1}\}}\approx \mathbf h\mathbf p(xy)^n\mathbf c^\prime\\
\stackrel{\mathbf h\mathbf p\mathbf c\approx\mathbf h\mathbf q\mathbf c}\approx{}& \mathbf h\mathbf q(xy)^n\mathbf c^\prime\stackrel{xty^2x\approx xtxy^2x}\approx \mathbf hyx^2y\mathbf c^\prime\stackrel{\{\tau_m,\,\eqref{xi_m+1}\}}\approx \mathbf h\mathbf q\mathbf c^\prime.
\end{aligned}
\]
Evidently, $\mathbf h\mathbf p\mathbf c^\prime\approx \mathbf h\mathbf q\mathbf c^\prime$ implies $\mathbf h\mathbf p\mathbf c\approx \mathbf h\mathbf q\mathbf c$.
Thus, $\mathbf h\mathbf p\mathbf c\approx \mathbf h\mathbf q\mathbf c$ is equivalent modulo $\Sigma$ to $\mathbf h\mathbf p\mathbf c^\prime\approx \mathbf h\mathbf q\mathbf c^\prime$.
It follows that each variety from the interval $[\mathbf M(xy),\mathbf V]$ can be defined by the identities in $\Sigma\cup\Phi_1\cup\overline{\Phi_2}$, where
\[
\overline{\Phi_2}:=\left\{ \; \mathbf h\mathbf p\mathbf c \approx \mathbf h\mathbf q\mathbf c \; \; \left| 
\begin{array}{l} 
\mathbf h \in \{ 1,\, yh\}; \; \alf(\mathbf p)=\alf(\mathbf q)=\{x,y\}; \\[0.05in] 
\mathbf c \in \big\{ 1,\, txy,\, \prod_{i=1}^k (t_i\mathbf a_i), \,\prod_{i=2}^{k+1} (t_i\mathbf a_i) \,\big|\, 1\le k \le m+2\big\}; \\[0.05in] 
1 \leq |\mathbf p|_x = |\mathbf q|_x, \, |\mathbf p|_y = |\mathbf q|_y \leq 2; \\[0.05in] 
|\mathbf h\mathbf p\mathbf c|_x = |\mathbf h\mathbf q\mathbf c|_x,\, |\mathbf h\mathbf p\mathbf c|_y = |\mathbf h\mathbf q\mathbf c|_y \geq 2 
\end{array} 
\right. \right\}.
\]
Since the set $\Sigma\cup\Phi_1\cup\overline{\Phi_2}$ is finite, the variety $\mathbf V$ is finitely based and the interval $[\mathbf M(xy),\mathbf V]$ is finite.
As shown above, any variety satisfying $\Sigma$ and not containing the variety $\mathbf M(xy)$ is a Cross variety and, in particular, small.
This implies that the variety $\mathbf V$ is small.
Now Lemma~\ref{L: xyx=xyxx is finitely generated} applies, yielding that $\mathbf V$ is a Cross variety.
\end{proof}

\section{Certain almost Cross varieties}
\label{sec: almost}

This section presents almost Cross varieties of monoids.
Results related to these varieties that are required later will also be established.

\subsection{The varieties satisfying the identity $xyxz\approx xyxzx$}
\label{subsec: xyxz=xyxzx}

Let
\[
\mathbf F:=\var\{xyxz\approx xyxzx,\,x^2y^2\approx y^2x^2\}.
\]
To describe subvarieties of $\mathbf F$, we need some notation.
For any $k \ge1$ and $1\le m \le k$, put
\[
\mathbf b_{k,m}:=x_{k-1}x_kx_{k-2}x_{k-1}\cdots x_{m-1}x_m.
\]
For brevity, we denote $\mathbf b_{k,1}$ simply by $\mathbf b_k$.
We also set $\mathbf{b}_0 := 1$ for convenience.
We introduce the following four countably infinite series of identities:
\[
\begin{aligned}
\alpha_k:&\enskip x_ky_kx_{k-1}x_ky_k \mathbf b_{k-1} \approx y_kx_kx_{k-1}x_ky_k \mathbf b_{k-1},\\
\beta_k:&\enskip xx_kx\mathbf b_k \approx x_kx^2\mathbf b_k,\\
\gamma_k:&\enskip y_1y_0x_ky_1 \mathbf b_k \approx y_1y_0y_1x_k \mathbf b_k,\\
\delta_k^m:&\enskip y_{m+1}y_mx_ky_{m+1} \mathbf b_{k,m}y_m\mathbf b_{m-1}\approx y_{m+1}y_my_{m+1}x_k\mathbf b_{k,m}y_m\mathbf b_{m-1},
\end{aligned}
\]
where $k\ge1$ and $1\le m\le k$.
The trivial variety of monoids is denoted by $\mathbf T$.
Let $\mathbf D$ denote the variety generated by the monoid
\[
D:=\langle a,b\mid a^2=ba=0,\,ab = a,\,b^2=b\rangle\cup\{1\}=\{a,b,0,1\}.
\]
It is verified in~\cite[the dual of Proposition~3.1(a)]{Edmunds-77} that
\[
\mathbf D=\var\{x^2\approx x^3,\,x^2y\approx xyx,\,x^2y^2\approx y^2x^2\}.
\]
The subvariety lattice of a monoid variety $\mathbf V$ is denoted by $\mathfrak L(\mathbf V)$.

\begin{proposition}[\!{\cite[Proposition~6.1]{Gusev-Vernikov-18}}]
\label{P: F}
\noindent
\begin{itemize}
\item[\textup{(i)}] The variety $\mathbf F$ is a non-finitely generated almost Cross variety.
\item[\textup{(ii)}] The lattice $\mathfrak L(\mathbf F)$ forms the following chain:
\begin{align*}
&\mathbf M(\emptyset)\subset \mathbf M(1)\subset \mathbf M(x)\subset\mathbf M(xy)\subset \mathbf D\\
\subset{}&\mathbf F\{\alpha_1\}\subset\mathbf F\{\beta_1\}\subset\mathbf F\{\gamma_1\}\subset\mathbf F\{\delta_1\}\\
\subset{}&\mathbf F\{\alpha_2\}\subset\mathbf F\{\beta_2\}\subset\mathbf F\{\gamma_2\}\subset\mathbf F\{\delta_2\}\subset\mathbf F\{\delta_2^2\}\\
\rule{6pt}{0pt}\vdots&\\
\subset{}&\mathbf F\{\alpha_k\}\subset\mathbf F\{\beta_k\}\subset\mathbf F\{\gamma_k\}\subset\mathbf F\{\delta_k\}\subset\mathbf F\{\delta_k^2\}\subset\cdots\subset\mathbf F\{\delta_k^k\}\\
\rule{6pt}{0pt}\vdots&\\
\subset{}&\mathbf F.
\end{align*}
\end{itemize}
\end{proposition}

\begin{lemma}[\!{\cite[Lemma~3.2]{Gusev-Li-Zhang-25}}]
\label{L: does not contain LRB,F_1}
Let $\mathbf V$ be a variety satisfying the identity $xyxz\approx xyxzx$.
\begin{itemize}
\item[\textup{(i)}] If $\mathbf B_2\nsubseteq\mathbf V$, then $\mathbf V\subseteq\mathbf F$.
\item[\textup{(ii)}] If $\mathbf F\{\alpha_1\}\nsubseteq\mathbf V$, then $\mathbf V\subseteq\mathbf B_2\vee \mathbf M(x)=\var\{x^2 \approx x^3,\,x^2y \approx xyx\}$.
\end{itemize}
\end{lemma}

\begin{lemma}[\!{\cite[Lemma~4.2]{Gusev-Li-Zhang-25}}]
\label{L: kappa_k and delta_k^k}
For any $k \ge1$, the identities $\delta_k^k$ and
\[
\kappa_k:\enskip xx_kx\mathbf b_k \approx x^2x_k\mathbf b_k
\] 
are equivalent modulo $xyxz\approx xyxzx$.
\end{lemma}

We introduce six more countably infinite series of identities:
\[
\begin{aligned}
\varepsilon_{k-1}:&\enskip y_kx_{k-1}xy_kx\mathbf b_{k-1} \approx y_kx_{k-1}y_kx^2\mathbf b_{k-1},\\
\zeta_{k-1}:&\enskip x y_0 y x_{k-1} xy \mathbf b_{k-1} \approx x y_0 y x_{k-1} yx \mathbf b_{k-1},\\
\lambda_k^m:&\enskip x y_m y x_k xy \mathbf b_{k,m}y_m\mathbf b_{m-1}\approx x y_m y x_k yx \mathbf b_{k,m}y_m\mathbf b_{m-1},\\
\eta_{k-1}:&\enskip x y_1 y y_0 xy x_{k-1}y_1\mathbf b_{k-1}\approx x y_1 y y_0 yxx_{k-1}y_1\mathbf b_{k-1},\\
\mu_k^m:&\enskip x y_{m+1} y y_m xy x_ky_{m+1}\mathbf b_{k,m}y_m\mathbf b_{m-1}\approx x y_{m+1} y y_m yxx_k y_{m+1}\mathbf b_{k,m}y_m\mathbf b_{m-1},\\
\nu_{k-1}:&\enskip x x_{k-1}yzxy h^2z \mathbf b_{k-1}\approx x x_{k-1}yzyx h^2z \mathbf b_{k-1},
\end{aligned}
\]
where $k \ge1$ and $1\le m \le k$.

\begin{lemma}
\label{L: O_k}
If a variety $\mathbf V$ of monoids satisfies the identities
\begin{align}
\label{xyxz=xyxzx}
xyxz\approx {}&xyxzx,\\
\label{xzytxyhhz=xzytyxhhz}
xzytxyh^2z\approx {}&xzytyxh^2z
\end{align}
and $\kappa_k$ for some $k\ge1$, then $\mathbf V$ is a Cross variety.
\end{lemma}

\begin{proof}
Consider an arbitrary subvariety $\mathbf X$ of the variety $\mathbf V$.
It the following, we will show that $\mathbf X$ can be given by the identities in
\[
\begin{aligned}
&\Psi_0:=\{xyxz\approx xyxzx,\, xzytxyh^2z\approx xzytyxh^2z,\,\kappa_k\},\\
&\Psi_1:=\{x\approx y, x\approx x^2, x^2\approx x^3, xy\approx yx, x^2y\approx xyx, x^2y\approx yx^2, x^2y^2\approx y^2x^2,xy\approx xyx\},\\
&\Psi_2:=\{\alpha_s,\,\beta_s,\,\gamma_s,\,\delta_s^t,\,\varepsilon_{s-1},\,\zeta_{s-1},\,\lambda_s^t,\,\eta_{s-1},\,\mu_k^t\mid 1\le t \le s \le k+1\}.
\end{aligned}
\]
Since the variety $\mathbf X$ is arbitrary and the set $\Psi_0\cup\Psi_1\cup\Psi_2$ is finite, the variety $\mathbf V$ must be finitely based and small.
Now Lemma~\ref{L: xyx=xyxx is finitely generated} applies, yielding that $\mathbf V$ is a Cross variety.

\smallskip

If $\mathbf B_2\nsubseteq\mathbf X$, then $\mathbf X\subseteq\mathbf F$ by~\cite[Lemma~3.2(i)]{Gusev-Li-Zhang-25}. 
In this case, since $\kappa_k$ holds in $\mathbf X$, it follows from Proposition~\ref{P: F} and Lemma~\ref{L: kappa_k and delta_k^k} that $\mathbf X$ may be defined by the identities in $\Psi_0 \cup \Psi_1 \cup \Psi_2$. 

If $\mathbf F\{\alpha_1\}\nsubseteq\mathbf X$, then $\mathbf X\subseteq\mathbf B_2\vee\mathbf M(x)$ by~\cite[Lemma~3.2(ii)]{Gusev-Li-Zhang-25}. 
In this case, it follows from the description of the subvariety lattice $\mathfrak L(\mathbf B_2\vee\mathbf M(x))$ obtained in~\cite[Proposition~4.1]{Lee-12b} that $\mathbf X$ may be defined by the identities in $\Psi_0\cup\Psi_1$.

Now assume that $\mathbf B_2 \vee \mathbf F\{\alpha_1\}\subseteq\mathbf X$.
Since
\[
xtyxy\stackrel{\eqref{xyxz=xyxzx}}\approx xtyxy x^{2k} \stackrel{\kappa_k}\approx xty^2x^{2k+1}\stackrel{\eqref{xyxz=xyxzx}}\approx xty^2x,
\]
we see that $\mathbf V$ and so $\mathbf X$ satisfy the identity $xtyxy\approx xty^2x$.
It is verified in~\cite[Corollary~4.18]{Gusev-Li-Zhang-25} that every variety of monoids satisfying the identities
\[
xyxz\approx xyxzx,\quad xzytxyh^2z\approx xzytyxh^2z,\quad xtyxy\approx xty^2x
\]
and containing $\mathbf B_2 \vee \mathbf F\{\alpha_1\}$ may be given by the identities in $\Psi_0$ together with the identities in 
\[
\Psi := \{(xy)^2\approx x^2y^2,\,\gamma_k,\,\delta_k^m,\,\varepsilon_{k-1},\,\zeta_{k-1},\,\lambda_k^m,\,\eta_{k-1},\,\mu_k^m,\,\nu_{k-1} \mid k \ge1, \ 1 \le m \le k \}.
\]
In particular, $\mathbf X$ is defined by a subset of $\Psi_0\cup\Psi$.
We are going to refine this set of identities and verify that $\mathbf X$ may be defined by identities in $\Psi_0\cup\Psi_2$.

First, notice that $\mathbf B_2$ and so $\mathbf X$ violate $\alpha_i$ and $\beta_i$ for each $i\ge1$ by Lemma~\ref{L: word problem LRB}, while the identity $\nu_{s-1}$ with $s\ge1$ holds in $\mathbf V$ because
\[
\begin{aligned}
&x x_{s-1}yzxy h^2z \mathbf b_{s-1}\stackrel{\eqref{xyxz=xyxzx}}\approx x x_{s-1}y(zx)y h^2(zx) h^{4k-4} \mathbf b_{s-1} \stackrel{\kappa_k}\approx x x_{s-1}y^2(zx) h^2(zx) h^{4k-4} \mathbf b_{s-1}\\
\stackrel{\eqref{xyxz=xyxzx}}\approx {}&x x_{k-1}y^2z(x h^2)z (x h^2)^{2k-2} \mathbf b_{k-1} \stackrel{\kappa_k}\approx x x_{k-1}yzy(x h^2)z (x h^2)^{2k-2} \mathbf b_{k-1} \stackrel{\eqref{xyxz=xyxzx}}\approx x x_{k-1}yzyx h^2z \mathbf b_{k-1}.
\end{aligned}
\]
Take an arbitrary $r>k+1$.
It is proved in~\cite[Lemma 4.3(ii),(iii)]{Gusev-Li-Zhang-25} that the inclusions
\[
\begin{aligned}
\var\{xyxz\approx xyxzx, \gamma_{k+1}\}\subset{}&\var\{xyxz\approx xyxzx,\,\delta_p^s\}\subset\var\{xyxz\approx xyxzx,\,\delta_{p+q}^p\}\\
\subset{}&\var\{xyxz\approx xyxzx,\,\varepsilon_p\}\subset\var\{xyxz\approx xyxzx,\,(xy)^2\approx x^2y^2\}
\end{aligned}
\]
hold for each $p,q\ge1$ and $1\le s\le p$.
This fact and Lemma~\ref{L: kappa_k and delta_k^k} imply that the variety $\mathbf V$ satisfies the identities $(xy)^2\approx x^2y^2$, $\varepsilon_{r-1}$ and $\delta_r^t$ with $t\ge k$.
Further, $\mathbf V$ satisfy $\zeta_r$ because
\[
\begin{aligned}
&x y_0 y x_r xy \mathbf b_r \stackrel{\eqref{xyxz=xyxzx}}\approx x y_0 y (x_r x)y x_{r-1} (x_rx) \mathbf b_{r-1} \stackrel{\kappa_k}\approx x y_0 y^2 (x_r x) x_{r-1} (x_rx)\mathbf b_{r-1} \\
\stackrel{\eqref{xyxz=xyxzx}}\approx {}&x y_0 y^2 x_r (x x_{r-1}) x_r x_{r-2} (xx_{r-1})\mathbf b_{r-2} \stackrel{\kappa_k}\approx x y_0 y x_r y (x x_{r-1}) x_r x_{r-2} (xx_{r-1})\mathbf b_{r-2}\\
\stackrel{\eqref{xyxz=xyxzx}}\approx {}&x y_0 y x_{k-1} yx \mathbf b_{k-1}.
\end{aligned}
\]
By a similar argument, we can show that $\{\lambda_r^t\mid1 \le t \le r\}$ is satisfied by $\mathbf V$.
Since
\[
\begin{aligned}
&x y_1 y y_0 xy x_ry_1\mathbf b_r\stackrel{\eqref{xyxz=xyxzx}}\approx x y_1 y y_0 xy x_r(y_1x_{r-1})x_rx_{r-2}(y_1x_{r-1})\mathbf b_{r-2}\\
\stackrel{\kappa_k}\approx {}&x y_1 y y_0 xy x_r^2(y_1x_{r-1})x_{r-2}(y_1x_{r-1})\mathbf b_{r-2}\stackrel{\eqref{xzytxyhhz=xzytyxhhz}}\approx x y_1 y y_0 yx x_r^2(y_1x_{r-1})x_{r-2}(y_1x_k)\mathbf b_{r-2}\\
\stackrel{\kappa_k}\approx {}&x y_1 y y_0 yx x_r(y_1x_{r-1})x_rx_{r-2}(y_1x_{r-1})\mathbf b_{r-2}\stackrel{\eqref{xyxz=xyxzx}}\approx x y_1 y y_0 yx x_ry_1\mathbf b_r
\end{aligned}
\]
the identity $\eta_r$ holds in $\mathbf V$.
By a similar argument, we can show that $\mathbf V$ satisfies $\{\mu_r^t\mid 1\le t\le r\}$.
Finally, since 
\[
\begin{aligned}
&y_1y_0x_{k+1}y_1 \mathbf b_{k+1}\stackrel{\eqref{xyxz=xyxzx}}\approx y_1y_0x_{k+1}(y_1x_k)x_{k+1}x_{k-1}(y_1x_k) \mathbf b_k \stackrel{\kappa_k}\approx y_1y_0x_{k+1}^2(y_1x_k)x_{k-1}(y_1x_k) \mathbf b_k\\
\stackrel{\eqref{xyxz=xyxzx}}\approx {}&y_1y_0x_{k+1}^2y_1^{2r} \mathbf b_{k+1}\stackrel{\kappa_k}\approx y_1y_0x_{k+1}y_1^2x_{k+1}y_1^{2r-2}x_{k}\mathbf b_{k}\stackrel{\gamma_r}\approx y_1y_0y_1x_{k+1}y_1x_{k+1}y_1^{2r-2}x_{k}\mathbf b_{k}\\
\stackrel{\eqref{xyxz=xyxzx}}\approx {}&y_1y_0y_1x_{k+1}^2x_{k}\mathbf b_{k}\stackrel{\kappa_k}\approx y_1y_0y_1x_{k+1}x_{k}x_{k+1} \mathbf b_{k}\stackrel{\eqref{xyxz=xyxzx}}\approx y_1y_0y_1x_{k+1} \mathbf b_{k+1},
\end{aligned}
\]
we see that $\mathbf V\{\gamma_r\}$ implies $\gamma_{k+1}$.
Since $\gamma_r$ is a consequence of $\{xyxz\approx xyxzx, \gamma_{k+1}\}$, it follows that $\mathbf V\{\gamma_r\}=\mathbf V\{\gamma_{k+1}\}$.
By a similar argument, we can show that $\mathbf V\{\delta_r^t\}=\mathbf V\{\delta_{k+1}^t\}$ for each $1\le t< k$.
Since $r$ is arbitrary, we have proved that $\mathbf X$ can be given by a subset of $\Psi_0\cup\Psi_2$.
Lemma~\ref{L: O_k} is thus proved.
\end{proof}

For an arbitrary $n \ge1$, we denote by $S_n$ the full symmetric group on the set $\{1,\dots,n\}$.
For arbitrary $n \ge1$ and $\pi,\pi^\prime\in S_{2n}$, we put
\[
\begin{aligned}
\mathbf w_n[\pi,\pi^\prime]&:=\biggl(\prod_{i=1}^n x_it_i\biggr) \,x\, \biggl(\prod_{i=1}^{2n} z_i\biggr) \,y\, \biggl(\prod_{i=n+1}^{2n} t_ix_i\biggr)\,txy\, \biggl(\prod_{i=1}^{2n} x_{i\pi}z_{i\pi^\prime}\biggr),\\
\mathbf w_n^\prime[\pi,\pi^\prime]&:=\biggl(\prod_{i=1}^n x_it_i\biggr) \,x\, \biggl(\prod_{i=1}^{2n} z_i\biggr) \,y\, \biggl(\prod_{i=n+1}^{2n} t_ix_i\biggr)\,tyx\, \biggl(\prod_{i=1}^{2n} x_{i\pi}z_{i\pi^\prime}\biggr).
\end{aligned}
\]
Let
\[
\mathbf G := \{xyxz\approx xyxzx,\,\kappa_1,\, \mathbf w_n[\pi,\pi^\prime] \approx \mathbf w_n^\prime[\pi,\pi^\prime] \mid n \ge1, \ \pi,\pi^\prime\in S_{2n}\}.
\]

\begin{lemma}[\!{\cite[Lemma~5.3]{Gusev-Li-Zhang-25}}]
\label{L: V does not contain G}
Let $\mathbf V$ be a variety satisfying $xyxz\approx xyxzx$. 
If $\mathbf G\nsubseteq\mathbf V$, then the identity $xzytxyh^2z\approx xzytyxh^2z$ holds in $\mathbf V$.
\end{lemma}

The following statement provides a new example of an almost Cross variety of aperiodic monoids.

\begin{proposition}
\label{P: G}
The variety $\mathbf G$ is a non-finitely based\textup, finitely generated almost Cross variety of monoids. 
\end{proposition}

\begin{proof}
It is proved in~\cite[Proposition~5.4]{Gusev-Li-Zhang-25} that the variety $\mathbf G$ is non-finitely based.
Therefore, $\mathbf G$ is a non-Cross variety.
According to Lemma~\ref{L: V does not contain G}, each proper subvariety of $\mathbf G$ satisfies the identity $xzytxyh^2z\approx xzytyxh^2z$.
In other words, the variety $\mathbf G\{xzytxyh^2z\approx xzytyxh^2z\}$ is a maximal subvariety of $\mathbf G$.
By the very definition, the identity $\kappa_1$ holds in $\mathbf G$.
Now Lemma~\ref{L: O_k} applies, yielding that $\mathbf G\{xzytxyh^2z\approx xzytyxh^2z\}$ is a Cross variety.
Therefore, $\mathbf G$ is an almost Cross variety with finitely many subvarieties.
Finally, by Lemma~\ref{L: xyx=xyxx is finitely generated}, the variety $\mathbf G$ is finitely generated.
\end{proof}

\begin{proposition}
\label{P: xyxz=xyxzx}
The varieties $\mathbf F$ and $\mathbf G$ are the only almost Cross varieties of monoids satisfying the identity $xyxz\approx xyxzx$.
\end{proposition}

\begin{proof}
The varieties $\mathbf F$ and $\mathbf G$ are almost Cross by Propositions~\ref{P: F} and~\ref{P: G}, respectively.
Let $\mathbf V$ be a variety satisfying the identity $xyxz\approx xyxzx$ and not containing the varieties $\mathbf F$ and $\mathbf G$.
Since $\mathbf F\nsubseteq\mathbf V$, Proposition~\ref{P: F} and Lemma~\ref{L: kappa_k and delta_k^k} imply that $\mathbf V$ satisfies $\kappa_k$ for some $k\ge1$. 
Since $\mathbf G\nsubseteq\mathbf V$, it follows from Lemma~\ref{L: V does not contain G} that $\mathbf V$ satisfies $xzytxyh^2z\approx xzytyxh^2z$. 
Now Lemma~\ref{L: O_k} applies, yielding that $\mathbf V$ is a Cross variety.
\end{proof}

\subsection{The varieties containing $\mathbf F\{\alpha_1\}\vee\mathbf D$}

Let 
\[
\mathbf I:=\var
\left\{
\left.
\begin{array}{l}
xyx\approx xyx^2,\,xyxztx\approx xyxzxtx,\,xyzxy\approx yxzxy,\,x^2y^2\approx y^2x^2,\\
xz_{1\pi} \cdots z_{n\pi}x t_1z_1\cdots t_nz_n\approx
x^2z_{1\pi} \cdots z_{n\pi}t_1z_1\cdots t_nz_n
\end{array}
\right|
\begin{array}{l}
n\ge1,\\
\pi\in S_n
\end{array}
\right\}.
\]
Note that a 31-element generator for the variety $\mathbf I$ was found in~\cite[Theorem~7.2(v)]{Sapir-21}. 

\begin{proposition}[\!{\cite[Corollary~1.2]{Gusev-20}}]
\label{P: I}
The variety $\mathbf I$ is a non-finitely based almost Cross variety.
\end{proposition}

\begin{lemma}[\!{\cite[Lemma~4.1]{Gusev-21}}]
\label{L: does not contain I}
Let $\mathbf V$ be a monoid variety that contains the variety $\mathbf F\{\alpha_1\}\vee\overleftarrow{\mathbf D}$ and satisfies the identity $xyx\approx xyx^2$. 
If $\mathbf I\nsubseteq\mathbf V$, then $\mathbf V$ satisfies $xzyxty\approx xzxyxty$.
\end{lemma}

Let $\mathbf K$ denote the variety generated by the monoid
\begin{align*}
K&:=\langle a,b,c\mid a^2=a,\,b^2=b^3,\,abc=ac=ba=b^2c=0,\,bcb^2=bcb,\,ca=c\rangle\cup\{1\}\\
&=\{a,b,c,ab,ab^2,b^2,bc,bcb,cb,cb^2,0,1\}.
\end{align*}

\begin{proposition}[\!{\cite[Propositions~5.1 and~6.5]{Gusev-Sapir-22}}]
\label{P: K1}
The variety $\mathbf K$ is a non-finitely based\textup, finitely generated almost Cross variety.
\end{proposition}

\begin{lemma}[\!{\cite[Lemma~4.3]{Gusev-Sapir-22}}]
\label{L: does not contain K1} 
Let $\mathbf V$ be a monoid variety that contains the variety $\mathbf F\{\alpha_1\}$ and satisfies the identity $xyx \approx xyx^2$.
If $\mathbf V$ does not contain $\mathbf K$, then $\mathbf V$ satisfies $xty^2x \approx xty(yx)^2$.
\end{lemma}

Let
\[
\mathbf P:=\var\{xyx\approx xyx^2,\,x^2y^2\approx y^2x^2,\,xyzxy\approx yxzxy,\,x^2yty\approx xyxty\approx yx^2ty\}.
\]

\begin{proposition}[\!{\cite[Proposition~3.1]{Gusev-25IJAC}}]
\label{P: P}
The variety $\mathbf P$ is a non-finitely generated almost Cross variety.\qed
\end{proposition}

\begin{lemma}[\!{\cite[Proposition~3.1 and Lemma~3.3]{Gusev-25IJAC}}]
\label{L: does not contain P}
Let $\mathbf V$ be a monoid variety such that $\mathbf F\{\alpha_1\}\vee\mathbf Q\subseteq\mathbf V$ satisfying the identity $xyx\approx xyx^2$.
If $\mathbf P\nsubseteq\mathbf V$, then $\mathbf V$ satisfies the identity $\sigma_m$ for some $m\ge2$.
\end{lemma}

\subsection{The varieties satisfying the identities $xyx\approx x^2yx\approx xyx^2$ and $x^2y^2\approx (xy)^2$}
\label{subsec: xyx=xxyx=xyxx,xxyy=xyxy}

Let $\mathbf E$ denote the variety generated by the monoid 
\[
E := \langle a, b, c \mid a^2 = ab = 0, ba = ca = a, b^2 = bc = b, c^2= cb = c \rangle\cup\{1\}=\{a,b,c,ac,0,1\}.
\]
It is shown in~\cite[Section~14]{Lee-Li-11} that 
\[
\mathbf E:=\var\{xyx\approx x^2yx\approx xyx^2,\, xy^2x \approx x^2y^2\}.
\]
Put
\[
\mathbf R:=\mathbf E\{xyzxy \approx xyzyx\}\ \text{ and }\ \mathbf R_m:=\mathbf E\{\overleftarrow{\sigma_m}\},
\]
where
\[
\overleftarrow{\sigma_m}:\enskip   \biggl(\prod_{i=m}^1 \mathbf a_it_i\biggr) xy \approx  \biggl(\prod_{i=m}^1 \mathbf a_it_i\biggr)yx,\ m\ge 2.
\]

\begin{proposition}[\!{\cite[Proposition~5.11]{Jackson-Lee-18}}]
\label{P: E1}
\noindent
\begin{itemize}
\item[\textup{(i)}] The lattice $\mathfrak L(\mathbf E)$ has the form shown in Fig.~\ref{pic: L(E1)}.
\item[\textup{(ii)}] The variety $\mathbf R$ is a non-finitely generated almost Cross variety.
\end{itemize}
\end{proposition}

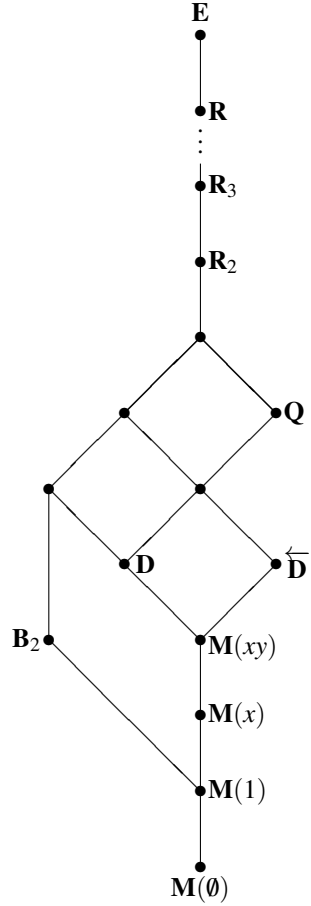
\begin{figure}[htb]
\unitlength=1mm
\linethickness{0.4pt}
\begin{center}
\begin{picture}(46,113)
\put(10,33){\circle*{1.33}}
\put(10,53){\circle*{1.33}}
\put(20,43){\circle*{1.33}}
\put(20,63){\circle*{1.33}}
\put(30,3){\circle*{1.33}}
\put(30,13){\circle*{1.33}}
\put(30,23){\circle*{1.33}}
\put(30,33){\circle*{1.33}}
\put(30,53){\circle*{1.33}}
\put(30,73){\circle*{1.33}}
\put(30,83){\circle*{1.33}}
\put(30,93){\circle*{1.33}}
\put(30,103){\circle*{1.33}}
\put(30,113){\circle*{1.33}}
\put(40,43){\circle*{1.33}}
\put(40,63){\circle*{1.33}}
\put(20,63){\line(1,1){10}}
\put(30,73){\line(1,-1){10}}

\put(20,63){\line(1,1){10}}
\put(30,73){\line(1,-1){10}}
\put(40,63){\line(-1,-1){20}}
\put(30,3){\line(0,1){30}}
\put(30,33){\line(-1,1){20}}
\put(10,53){\line(0,-1){20}}
\put(10,33){\line(1,-1){20}}
\put(30,33){\line(1,1){10}}
\put(40,43){\line(-1,1){20}}
\put(20,63){\line(-1,-1){10}}
\put(30,73){\line(0,1){23}}
\put(30,103){\line(0,1){10}}
\put(30,100){\makebox(0,0)[cc]{$\vdots$}}

\put(31,23){\makebox(0,0)[lc]{$\mathbf M(x)$}}
\put(31,32){\makebox(0,0)[lc]{$\mathbf M(xy)$}}
\put(30,116){\makebox(0,0)[cc]{$\mathbf E$}}
\put(31,83){\makebox(0,0)[lc]{$\mathbf R_2$}}
\put(31,93){\makebox(0,0)[lc]{$\mathbf R_3$}}
\put(31,103){\makebox(0,0)[lc]{$\mathbf R$}}
\put(21.5,43){\makebox(0,0)[lc]{$\mathbf D$}}
\put(41,43){\makebox(0,0)[lc]{$\overleftarrow{\mathbf D}$}}
\put(9,33){\makebox(0,0)[rc]{$\mathbf B_2$}}
\put(41,63){\makebox(0,0)[lc]{$\mathbf Q$}}
\put(31,13){\makebox(0,0)[lc]{$\mathbf M(1)$}}
\put(30,0){\makebox(0,0)[cc]{$\mathbf M(\emptyset)$}}
\end{picture}
\end{center}
\caption{The lattice $\mathfrak L(\mathbf E)$}
\label{pic: L(E1)}
\end{figure}

\begin{lemma}[\!{\cite[Fact~5.3]{Gusev-Sapir-26}}]
\label{L: nsub B2} 
The variety $\mathbf E$ is the largest variety satisfying the identities $xyx\approx x^2yx\approx xyx^2$ and $x^2y^2\approx (xy)^2$ which does not contain $\overleftarrow{\mathbf B_2}$.
\end{lemma}

\begin{lemma}[\!{\cite[Observation~5.8, Fact~6.2 and the dual results]{Gusev-Sapir-26}}]
\label{L: E23}
Let $\mathbf V$ be a variety such that $\mathbf Q \vee \mathbf B_2 \vee \overleftarrow{\mathbf B_2} \subseteq \mathbf V \subseteq \mathbf E \vee \overleftarrow{\mathbf E}$. 
Then $\mathbf V$ can be defined within $\mathbf E \vee \overleftarrow{\mathbf E}$ by the identities 
\[
\begin{aligned}
&\tau_m:\enskip  x^2y^2\biggl(\prod_{i=1}^m t_i\mathbf a_i\biggr)  \approx  xy^2x\biggl(\prod_{i=1}^m t_i\mathbf a_i\biggr),\ m\ge1,\\
&\tau_\infty:\enskip xyzxy\approx xyxzxy,\\
&\overleftarrow{\tau_m}:\enskip  \biggl(\prod_{i=m}^1 \mathbf a_it_i\biggr) x^2y^2 \approx  \biggl(\prod_{i=m}^1 \mathbf a_it_i\biggr) yx^2y,\ m\ge1,\\
&\overleftarrow{\tau_\infty}:\enskip xyzxy\approx xyzyxy.
\end{aligned}
\]
\end{lemma}

For each $m\ge1$, we introduce two more countably infinite series of identities:
\[
\begin{aligned}
&\chi_m:\enskip  (xy)^2\biggl(\prod_{i=1}^m t_i\mathbf a_i\biggr)  \approx (xy)^2x\biggl(\prod_{i=1}^m  t_i\mathbf a_i\biggr),
&\overleftarrow{\chi_m}:\enskip  \biggl(\prod_{i=m}^1 \mathbf a_it_i\biggr) (xy)^2 \approx  \biggl(\prod_{i=m}^1 \mathbf a_it_i\biggr) y(xy)^2.
\end{aligned}
\]

\begin{lemma}
\label{L: does not contain Rn}
Let $\mathbf V$ be a variety such that $M(xy)\in\mathbf V$ satisfying the identities $xyx\approx x^2yx\approx xyx^2$. 
If $\mathbf R_{m+1}\nsubseteq\mathbf V$ for some $m\ge 2$, then $\mathbf V$ satisfies $\overleftarrow{\chi_m}$.
\end{lemma}

\begin{proof}
First, suppose that $Q\notin\mathbf V$.
Then $\mathbf V$ satisfies $x^2yzx\approx x^2yxzx$ by Lemma~\ref{L: Q nsubseteq V}.
Since
\[
\mathbf c_1 (xy)^2 \stackrel{xyx\approx x^2yx}\approx \mathbf c_2 (xy)^2\stackrel{x^2yzx\approx x^2yxzx}\approx \mathbf c_2 y(xy)^2\stackrel{xyx\approx x^2yx}\approx \mathbf c_1 y(xy)^2,
\]
where $\mathbf c_1:=\mathbf a_mt_m\cdots \mathbf a_1t_1$ and $\mathbf c_2:=\mathbf a_m^2t_m\cdots \mathbf a_1^2t_1$, we see that $\mathbf V$ satisfies $\overleftarrow{\chi_m}$.
Now suppose that $Q\in\mathbf V$.
It follows from~\cite[Example~4.2]{Sapir-23} and Lemma~\ref{L: id Q1} that $\mathbf V$ satisfies an identity $\mathbf a_m^{p_m}t_m\cdots \mathbf a_1^{p_1}t_1\mathbf p\approx\mathbf a_m^{q_m}t_m\cdots \mathbf a_1^{q_1}t_1\mathbf q$ with $\ini(\mathbf p)=xy$, $\ini(\mathbf q)=yx$ and $p_1,q_1,\dots,p_m,q_m\ge 1$.
Then the identity $\mathbf a_m^{p_m}t_m\cdots \mathbf a_1^{p_1}t_1\mathbf p(xy)^2\approx\mathbf a_m^{q_m}t_m\cdots \mathbf a_1^{q_1}t_1\mathbf q(xy)^2$ is equivalent modulo $xyx\approx x^2yx\approx xyx^2$ to $\overleftarrow{\chi_m}$ as required.
\end{proof}

\begin{proposition}
\label{P: xyx=xxyx=xyxx,xxyy=xyxy}
The varieties $\mathbf B$, $\mathbf R$ and  $\overleftarrow{\mathbf R}$ are the only almost Cross varieties of monoids satisfying the identities $xyx\approx x^2yx\approx xyx^2$ and $x^2y^2\approx (xy)^2$.
\end{proposition}

\begin{proof}
The varieties $\mathbf B$, $\mathbf R$ and $\overleftarrow{\mathbf R}$ are almost Cross by Propositions~\ref{P: B} and~\ref{P: E1}(ii), and the result dual of Proposition~\ref{P: E1}(ii).
Let $\mathbf V$ be a variety satisfying the identities $xyx\approx x^2yx\approx xyx^2$ and $x^2y^2\approx (xy)^2$, and not containing the varieties $\mathbf B$, $\mathbf R$ and $\overleftarrow{\mathbf R}$.
If the variety $\mathbf V$ is either commutative or idempotent, then it is Cross by Corollary~\ref{C: com id Cross}.
Therefore, suppose that $\mathbf V$ is neither commutative nor idempotent, whence $M(xy)\in\mathbf V$ by Lemmas~\ref{L: isoterm} and~\ref{L: xy isoterm}.

If $\mathbf B_2\nsubseteq\mathbf V$, then $\mathbf V\subseteq\overleftarrow{\mathbf E}$ by the dual of Lemma~\ref{L: nsub B2}. 
In this case, since $\overleftarrow{\mathbf R}\nsubseteq\mathbf V$, the dual of Proposition~\ref{P: E1} implies $\mathbf V$ is a Cross variety.
Dually, if $\overleftarrow{\mathbf B_2}\nsubseteq\mathbf V$, then $\mathbf V$ is a Cross variety.
If $Q\notin \mathbf V$, then $\mathbf V$ satisfies $x^2yzx\approx x^2yxzx$ by Lemma~\ref{L: Q nsubseteq V}.
In this case, $\mathbf V\subseteq\mathbf O$ because
\[
\begin{aligned}
xzyxty\stackrel{xyx\approx x^2yx}\approx x^2zyxty\stackrel{x^2yzx\approx x^2yxzx}\approx x^2zxyxty\stackrel{xyx\approx x^2yx}\approx xzxyxty,\\
xzytyx\stackrel{xyx\approx x^2yx}\approx x^2zytyx\stackrel{x^2yzx\approx x^2yxzx}\approx x^2zytxyx\stackrel{xyx\approx x^2yx}\approx xzytxyx,
\end{aligned}
\]
whence $\mathbf V$ is a Cross variety by Lemma~\ref{L: O}.
So, we may further assume that $\mathbf Q\vee\mathbf B_2\vee\overleftarrow{\mathbf B_2}\subseteq\mathbf V$.

It is shown in~\cite[Fact~8.5]{Gusev-Sapir-26} that $\mathbf E\subseteq\mathbf M(x)\vee\overleftarrow{\mathbf B_3}$.
Since $\mathbf R\nsubseteq\mathbf V$, this implies that $\overleftarrow{\mathbf B_3}\nsubseteq\mathbf V$.
By a similar argument, we can show that $\mathbf B_3\nsubseteq\mathbf V$.
Then 
\[
\mathbf V\subseteq \mathbf E\vee \overleftarrow{\mathbf E}=\var\{xyx\approx x^2yx\approx xyx^2,\,x^2y^2\approx (xy)^2,\,xy^2z^2x\approx xy^2xz^2x \}
\] 
by~\cite[Corollary~5.6]{Gusev-Sapir-26}.
It follows from Proposition~\ref{P: E1}(i) that $\mathbf R_{m+1},\overleftarrow{\mathbf R_{m+1}}\nsubseteq\mathbf V$ for some $m\ge 2$.
Then $\mathbf V$ satisfies $\{\chi_m,\overleftarrow{\chi_m}\}$ by Lemma~\ref{L: does not contain Rn} and the dual result. 
Since $y(xy)^2\stackrel{x^2y^2\approx (xy)^2}\approx yx^2y^2\stackrel{xyx\approx xyx^2}\approx yx^2y$, we see that $\{\tau_m,\overleftarrow{\tau_m}\}$ is satisfied by $\mathbf V$.
Then $\mathbf V$ satisfies $\{\tau_s,\overleftarrow{\tau_s}\mid m\le s\le \infty\}$ as well.
This fact and Lemma~\ref{L: E23} imply that each variety from the interval $[\mathbf Q\vee\mathbf B_2\vee\overleftarrow{\mathbf B_2},\mathbf V]$ can be defined by the identities in $\{\tau_s,\overleftarrow{\tau_s}\mid 1\le s\le m\}$.
As shown above, any variety satisfying the identities $xyx\approx x^2yx\approx xyx^2$ and $x^2y^2\approx (xy)^2$ and not containing the varieties $\mathbf B$, $\mathbf Q\vee\mathbf B_2\vee\overleftarrow{\mathbf B_2}$, $\mathbf R$ and $\overleftarrow{\mathbf R}$ is a Cross variety and, in particular, small.
Then the variety $\mathbf V$ is small and finitely based.
Now Lemma~\ref{L: xyx=xyxx is finitely generated} applies, yielding that $\mathbf V$ is a Cross variety.
\end{proof}

\subsection{The varieties satisfying the identities $xyx\approx x^2yx$ and $xty^2x\approx xtxy^2x$}
\label{subsec: xyx=xxyx,xtyyx=xtxyyx}

Let
\[
\begin{aligned}
&\mathbf H:=\var\{xyx\approx xyx^2\approx x^2yx,\,(xy)^2\approx (yx)^2,\,xty^2x\approx xtxy^2x\},\\  &\mathbf H_m:=\mathbf H\{\sigma_m\},\ m\ge 2.
\end{aligned}
\]  

\begin{proposition}[\!{\cite[Proposition~3.3]{Gusev-Lee-Zhang-26}}]
\label{P: H}
\noindent
\begin{itemize}
\item[\textup{(i)}] The variety $\mathbf H$ is a non-finitely generated almost Cross variety of monoids. 
\item[\textup{(ii)}] The lattice $\mathfrak L(\mathbf H)$ has the form shown in Fig.~\ref{pic: L(H)}.
\end{itemize}
\end{proposition}

\begin{figure}[htb]
\unitlength=1mm
\linethickness{0.4pt}
\begin{center}
\begin{picture}(70,105)
\put(35,5){\circle*{1.33}}
\put(35,15){\circle*{1.33}}
\put(35,25){\circle*{1.33}}
\put(35,35){\circle*{1.33}}
\put(45,45){\circle*{1.33}}
\put(25,45){\circle*{1.33}}
\put(35,55){\circle*{1.33}}
\put(25,65){\circle*{1.33}}
\put(45,65){\circle*{1.33}}
\put(35,75){\circle*{1.33}}
\put(35,85){\circle*{1.33}}
\put(35,95){\circle*{1.33}}
\put(35,105){\circle*{1.33}}
\put(35,102){\makebox(0,0)[cc]{$\vdots$}}

\put(35,5){\line(0,1){30}}
\put(35,35){\line(-1,1){10}}
\put(35,35){\line(1,1){10}}
\put(25,45){\line(1,1){10}}
\put(25,45){\line(1,1){20}}
\put(25,65){\line(1,1){10}}
\put(45,45){\line(-1,1){20}}
\put(45,65){\line(-1,1){10}}
\put(35,75){\line(0,1){23}}

\put(35,2){\makebox(0,0)[cc]{$\mathbf M(\emptyset)$}}
\put(37,15){\makebox(0,0)[lc]{$\mathbf M(1)$}}
\put(37,25){\makebox(0,0)[lc]{$\mathbf M(x)$}}
\put(37,35){\makebox(0,0)[lc]{$\mathbf M(xy)$}}
\put(23,45){\makebox(0,0)[rc]{$\mathbf D$}}
\put(47,45){\makebox(0,0)[lc]{$\overleftarrow{\mathbf D}$}}
\put(23,65){\makebox(0,0)[rc]{$\mathbf Q$}}
\put(47,65){\makebox(0,0)[lc]{$\mathbf A_0$}}
\put(37,85){\makebox(0,0)[lc]{$\mathbf H_2$}}
\put(37,95){\makebox(0,0)[lc]{$\mathbf H_3$}}
\put(37,105){\makebox(0,0)[lc]{$\mathbf H$}}
\end{picture}
\end{center}
\caption{The lattice $\mathfrak L(\mathbf H)$}
\label{pic: L(H)}
\end{figure}
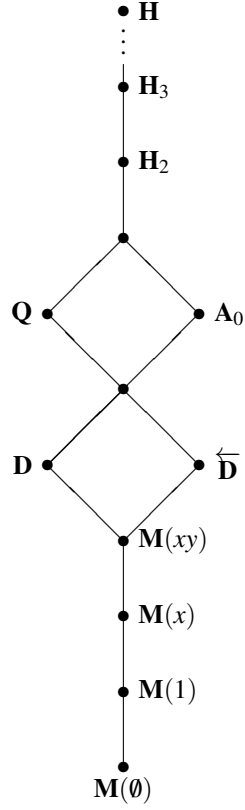

\begin{lemma}
\label{L: does not contain H}
Let $\mathbf V$ be a variety such that $M(xy)\in\mathbf V$ satisfying the identities $xyx\approx x^2yx$ and $xty^2x\approx xtxy^2x$. 
If $\mathbf H_{m+1}\nsubseteq\mathbf V$ for some $m\ge2$, then $\mathbf V$ satisfies 
\[
\xi_m:\enskip  x^2y^2 \biggl(\prod_{i=1}^m t_i\mathbf a_i\biggr) \approx (xy)^2 \biggl(\prod_{i=1}^m t_i\mathbf a_i\biggr).
\]
\end{lemma}

\begin{proof}
First, suppose that $Q\notin\mathbf V$.
Then $\mathbf V$ satisfies $x^2yzx\approx x^2yxzx$ by Lemma~\ref{L: Q nsubseteq V}.
Since
\[
x^2y^2\biggl(\prod_{i=1}^m t_i\mathbf a_i\biggr)\stackrel{x^2yzx\approx x^2yxzx}\approx x^2y^2xy\biggl(\prod_{i=1}^m t_i\mathbf a_i\biggr)\stackrel{xyx\approx x^2yx}\approx (xy)^2\biggl(\prod_{i=1}^m t_i\mathbf a_i\biggr),
\]
we see that $\mathbf V$ satisfies $\xi_m$.
Now suppose that $Q\in\mathbf V$.
Following~\cite{Sapir-21}, a set $\mathcal T$ of words is said to be \textit{stable with respect to a variety} $\mathbf X$ if $\mathbf v \in \mathcal T$ whenever $\mathbf u\in \mathcal T$ and $\mathbf X$ satisfies $\mathbf u \approx \mathbf v$.
It is routine to check that the set $\mathcal A:=\{x^sy^tt_1\mathbf a_1^{r_1}\cdots t_m\mathbf a_m^{r_m}\mid s,t,r_1\dots,r_m\ge 1\}$ is stable with respect to the variety $\mathbf H_{m+1}$.
However, the set $\mathcal A$ is not stable with with respect to the variety $\mathbf H_m$ because $\mathbf H_m$ satisfies $\sigma_m$.
Then it follows from~\cite[Corollary~3.6]{Sapir-21} that $\mathbf V$ satisfies an identity $\mathbf u\approx\mathbf v$ with $\mathbf u\in \mathcal A$ and $\mathbf v\notin \mathcal A$.
In particular, $\mathbf u=x^py^qt_1\mathbf a_1^{p_1}\cdots t_m\mathbf a_m^{p_m}$ for some $p,q,p_1,\dots,p_m\ge1$.
According to Lemma~\ref{L: id Q1}, $\mathbf v=\mathbf qt_1\mathbf a_1^{\ell_1}\cdots t_m\mathbf a_m^{\ell_m}$, where $\ell_1,\dots,\ell_m\ge1$ and the word $\mathbf q$ depends on the letters $x,y$ and contains $yx$ as a subword.
Then the identity $x^{p+1}y^{q+1}t_1\mathbf a_1^{p_1}\cdots t_m\mathbf a_m^{p_m}\approx x\mathbf qyt_1\mathbf a_1^{\ell_1}\cdots t_m\mathbf a_m^{\ell_m}$ is equivalent modulo $xyx\approx x^2yx\approx xyx^2$ to the identity $\xi_m$ as required.
\end{proof} 

\begin{proposition}
\label{P: xyx=xxyx,xtyyx=xtxyyx}
The varieties $\mathbf H$ and $\overleftarrow{\mathbf R}$ are the only almost Cross varieties of monoids satisfying the identities $xyx\approx x^2yx$ and $xty^2x\approx xtxy^2x$.
\end{proposition}

\begin{proof}
The varieties $\mathbf H$ and $\overleftarrow{\mathbf R}$ are almost Cross by Proposition~\ref{P: H}(i) and the dual of Proposition~\ref{P: E1}(ii). 
Let $\mathbf V$ be a variety satisfying the identities $xyx\approx x^2yx$ and $xty^2x\approx xtxy^2x$, and not containing the varieties $\mathbf H$ and $\overleftarrow{\mathbf R}$.
If $\mathbf V$ is commutative or idempotent, then, since $\mathbf B$ violates $xty^2x\approx xtxy^2x$ by Lemma~\ref{L: x=xx}, it follows from Corollary~\ref{C: com id Cross} that $\mathbf V$ is a Cross variety.
Therefore, suppose that $\mathbf V$ is neither commutative nor idempotent, whence $M(xy)\in\mathbf V$ by Lemmas~\ref{L: isoterm} and~\ref{L: xy isoterm}.
According to Proposition~\ref{P: H} and the dual of Proposition~\ref{P: E1}, $\mathbf H_{m+1},\overleftarrow{\mathbf R_{m+1}}\nsubseteq\mathbf V$ for some $m\ge2$.
Then $\mathbf V$ satisfies $\{\chi_m,\xi_m\}$ by Lemma~\ref{L: does not contain H} and the dual of Lemma~\ref{L: does not contain Rn}.
Since
\[
x^2y^2\biggl(\prod_{i=1}^m t_i\mathbf a_i\biggr)\stackrel{\xi_m}\approx (xy)^2\biggl(\prod_{i=1}^m t_i\mathbf a_i\biggr)\stackrel{\chi_m}\approx (xy)^2x\biggl(\prod_{i=1}^m t_i\mathbf a_i\biggr)\stackrel{xty^2x\approx xtxy^2x}\approx xy^2x\biggl(\prod_{i=1}^m t_i\mathbf a_i\biggr),
\]
we see that $\tau_m$ holds in $\mathbf V$ as well.
Now Lemma~\ref{L: O-2} applies, yielding that $\mathbf V$ is a Cross variety.
\end{proof}

\subsection{The varieties containing $\mathbf M(xyx)$}
\label{subsec: var M(xyx)}

Let
\[
\mathbf L:=\mathbf M(\{xt_1x\cdots t_nx\mid n\ge1\}), \ \mathbf Y_1:=\mathbf M(xzxyty),\ \mathbf Y_2:=\mathbf M(xyzxty,xtyzxy).
\]
Lee~\cite[Proposition~4.1]{Lee-14b} showed that
\[
\mathbf L=\var\{x^2\approx x^3,\, xyzxty\approx yxzxty,\, xzxyty\approx xzyxty,\, xzytxy\approx xzytyx\}.
\]

\begin{proposition}[\!{\cite[Proposition~5.1]{Jackson-05}}]
\label{P: Y_1 and Y_2}
The varieties $\mathbf Y_1$ and $\mathbf Y_2$ are non-finitely based\textup, finitely generated almost Cross varieties.
\end{proposition}

\begin{proposition}[\!{\cite[Theorem~2]{Lee-13}}]
\label{P: L}
The variety $\mathbf L$ is a non-finitely generated almost Cross variety.
\end{proposition}

\subsection{The varieties containing $\mathbf A_0$}
\label{subsec: A1}

Let $\mathbf A$ denote the variety generated by the monoid 
\[
A:=\langle a,b,c\mid a^2=a,\,b^2=b,\,ab=ca=0,\,ac=cb=c\rangle\cup\{1\}=\{a,b,c,ba,bc,0,1\}.
\]
In view of~\cite[Remark~3.7]{Gusev-Lee-Zhang-26}, we have $\mathbf A:=\mathbf H\{\sigma_3\}$.
Put
\[
\mathbf N:=\mathbf A\vee\mathbf R_2,\quad
\mathbf S:=\mathbf A_0\vee\mathbf R_2\vee \overleftarrow{\mathbf R_2},\quad
 \mathbf Z:=\mathbf A\vee\overleftarrow{\mathbf A}.
\]

\begin{lemma}[{\!\cite[proof of Lemma~4.5]{Sapir-23}}]
\label{L: E1{sigma2} or dA1 nsubseteq V}
Let $\mathbf V$ be a monoid variety that satisfies $xyx \approx x^2yx \approx xyx^2$ and contains neither $\mathbf R_2$ nor $\overleftarrow{\mathbf A}$. 
Then $\mathbf V$ satisfies the identity $xty^2x\approx xtxy^2x$.
\end{lemma}

The next two statements provide new examples of almost Cross varieties of aperiodic monoids.

\begin{proposition}
\label{P: N}
The variety $\mathbf N$ is a non-finitely based\textup, finitely generated almost Cross variety.
\end{proposition}

\begin{proof}
The variety $\mathbf N=\mathbf A\vee\mathbf R_2$ is finitely generated because the varieties $\mathbf A$ and $\mathbf R_2$ are so.
It is shown in~\cite[Theorem~4.6]{Sapir-23} that the variety $\mathbf N$ is non-finitely based, whence it is not a Cross variety.
Let $\mathbf V$ be a proper subvariety of $\mathbf N$.
Clearly, either $\mathbf A\nsubseteq\mathbf V$ or $\mathbf R_2\nsubseteq\mathbf V$.

Suppose that $\mathbf A\nsubseteq\mathbf V$.
Since $\overleftarrow{\mathbf R_2}$ is not a subvariety of $\mathbf N$, the dual of Lemma~\ref{L: E1{sigma2} or dA1 nsubseteq V} implies that $\mathbf V$ satisfies $xy^2tx\approx xy^2xtx$.
By the very definition, the variety $\mathbf R_2$ satisfies the identity $\overleftarrow{\sigma_2}$.
The variety $\mathbf A$ satisfies the identity $\overleftarrow{\sigma_2}$ as well because
\[
\mathbf cxy\stackrel{xyx\approx xyx^2}\approx \mathbf cx^4y^4\stackrel{xty^2x\approx xtxy^2x}\approx \mathbf c(x^2y^2)^2\stackrel{(xy)^2\approx (yx)^2}\approx\mathbf c(y^2x^2)^2\stackrel{xty^2x\approx xtxy^2x}\approx \mathbf cy^2x^2\stackrel{xyx\approx xyx^2}\approx \mathbf cyx,
\]
where $\mathbf c:=\mathbf a_2t_2\mathbf a_1t_1=yt_2xt_1$.
Now the duals of Lemmas~\ref{L: O} and~\ref{L: implies O} apply, yielding that $\mathbf V$ is a Cross variety.

Now suppose that $\mathbf R_2\nsubseteq\mathbf V$.
Since $\overleftarrow{\mathbf A}$ is not a subvariety of $\mathbf N$, Lemma~\ref{L: E1{sigma2} or dA1 nsubseteq V} implies that $\mathbf V$ satisfies $xty^2x\approx xtxy^2x$.
The variety $\mathbf A$ satisfies the identity $\sigma_3$ and, therefore, the identity $\tau_3$.
The variety $\mathbf R_2$ satisfies the identity $\tau_3$ as well because
\[
x^2y^2\mathbf d\stackrel{x^2y^2\approx xy^2x}\approx xy^2x\mathbf d\stackrel{xyx\approx xyx^2}\approx xy^2x^2\mathbf d\stackrel{x^2y^2\approx xy^2x}\approx xyx^2y\mathbf d\stackrel{xyx\approx xyx^2}\approx (xy)^2\mathbf d,
\]
where $\mathbf d:=t_1\mathbf a_1t_2\mathbf a_2t_3\mathbf a_3=t_1xt_2yt_3x$.
Now Lemma~\ref{L: O-2} applies, yielding that $\mathbf V$ is a Cross variety.

Thus, we have proved that $\mathbf V$ is a Cross variety in any case. 
Since $\mathbf V$ is an arbitrary proper subvariety of $\mathbf N$, it follows that $\mathbf N$ is an almost Cross variety.
\end{proof}

\begin{proposition}
\label{P: S}
The variety $\mathbf S$ is a non-finitely based\textup, finitely generated almost Cross variety.
\end{proposition}

\begin{proof}
First, notice that $\mathbf S=\mathbf A_0\vee\mathbf R_2\vee \overleftarrow{\mathbf R_2}$ satisfies $\tau_3$ and $\overleftarrow{\tau_3}$.
Indeed, $\mathbf A_0\vee\overleftarrow{\mathbf R_2}$ satisfy $\sigma_2$ which evidently implies $\tau_3$.
The identity  $\tau_3$ holds in $\mathbf R_2$ because it is a consequence of $xy^2x\approx x^2y^2$.
By the dual argument, we can show that $\mathbf S$ satisfies $\overleftarrow{\tau_3}$.

The variety $\mathbf S$ is finitely generated because the varieties $\mathbf A_0$, $\mathbf R_2$ and $\overleftarrow{\mathbf R_2}$ are so.
It is shown in~\cite[Theorem~7.5]{Gusev-Sapir-26} that the variety $\mathbf S$ is non-finitely based, whence it is not a Cross variety.
Let $\mathbf V$ be a proper subvariety of $\mathbf S$.
Clearly, either $\mathbf A_0\nsubseteq\mathbf V$ or $\mathbf R_2\nsubseteq\mathbf V$ or $ \overleftarrow{\mathbf R_2}\nsubseteq\mathbf V$.

Suppose that $\mathbf A_0\nsubseteq\mathbf V$.
Then $\mathbf V$ satisfies the identity $x^2y^2\approx (xy)^2$ by Lemma~\ref{L: A01 nsubseteq V}.
Since $\tau_3$ [respectively, $\overleftarrow{\tau_3}$] holds in $\mathbf S$ but does not hold in $\overleftarrow{\mathbf R}$ [respectively, $\mathbf R$], while $xy^2z^2x\approx xy^2xz^2x$ holds in $\mathbf S$ but does not hold in $\mathbf B$ by Lemma~\ref{L: x=xx}, the variety $\mathbf V$ does not contain $\mathbf B$, $\mathbf R$, or $\overleftarrow{\mathbf R}$ as subvarieties.
Now Proposition~\ref{P: xyx=xxyx=xyxx,xxyy=xyxy} applies, yielding that $\mathbf V$ is a Cross variety.

Now suppose that $\mathbf R_2\nsubseteq\mathbf V$.
Since $\overleftarrow{\mathbf A}$ is not a subvariety of $\mathbf S$, Lemma~\ref{L: E1{sigma2} or dA1 nsubseteq V} implies that $\mathbf V$ satisfies $xty^2x\approx xtxy^2x$.
Then $\mathbf V$ is a Cross variety by Lemma~\ref{L: O-2}.
By the dual argument, we can show that $\mathbf V$ is a Cross variety whenever $\overleftarrow{\mathbf R_2}\nsubseteq\mathbf V$.

Thus, we have proved that $\mathbf V$ is a Cross variety in any case. 
Since $\mathbf V$ is an arbitrary proper subvariety of $\mathbf S$, it follows that $\mathbf S$ is an almost Cross variety.
\end{proof}

\begin{proposition}[{\!\cite[Proposition~4.9]{Gusev-Lee-Zhang-26}}]
\label{P: Zhang-Luo}
The variety $\mathbf Z$ is a non-finitely based\textup, finitely generated almost Cross variety.
\end{proposition}

\section{Characterization of Cross varieties}
\label{sec: characterization}

\begin{theorem}
\label{T: main result}
A variety of aperiodic monoids is Cross if and only if it excludes every one of
the following almost Cross varieties: 
\begin{equation}
\label{D: 22 almost}
\mathbf B,\ \mathbf F,\ \overleftarrow{\mathbf F},\  \mathbf G,\  \overleftarrow{\mathbf G},\  \mathbf H,\  \overleftarrow{\mathbf H},\  \mathbf I, \ \overleftarrow{\mathbf I},\  \mathbf K,\  \overleftarrow{\mathbf K},\ \mathbf L,\ \mathbf N,\  \overleftarrow{\mathbf N},\  \mathbf P,\  \overleftarrow{\mathbf P},\ \mathbf R,\  \overleftarrow{\mathbf R},\ \mathbf S,\ \mathbf Y_1,\ \mathbf Y_2,\ \mathbf Z. 
\end{equation}
Consequently, these 22 varieties exhaust all almost Cross varieties of aperiodic monoids.
\end{theorem}

\begin{proof}
The variety $\mathbf B$ is almost Cross by Proposition~\ref{P: B}, while the other varieties in~\eqref{D: 22 almost} were shown to be almost Cross in Section~\ref{sec: almost}. Hence, all of them are excluded from any Cross variety.
Conversely, suppose that~$\mathbf V$ is a variety of aperiodic monoids, so that 
\begin{itemize}
\item[\textup{($\ast$)}] $\mathbf V$ contains none of the varieties from~\eqref{D: 22 almost}. 
\end{itemize}
If $\mathbf V$ is either commutative or idempotent, then it is Cross by Corollary~\ref{C: com id Cross}.
Therefore, suppose that $\mathbf V$ is neither commutative nor idempotent, whence $M(xy)\in\mathbf V$ by Lemmas~\ref{L: isoterm} and~\ref{L: xy isoterm}. 
Two cases are possible.

\medskip

\noindent\textsc{Case~1}:  $M(xyx)\in\mathbf V$. 
Since $\mathbf Y_1,\mathbf Y_2\nsubseteq\mathbf V$ by~($\ast$), Lemma~\ref{L: isoterm} implies that one of the following holds:
\begin{itemize}
\item $xyx$ is an isoterm for $\mathbf V$ and $xzxyty$ and $xzytxy$ are not isoterms for $\mathbf V$;
\item $xyx$ is an isoterm for $\mathbf V$ and $xzxyty$ and $xyzxty$ are not isoterms for $\mathbf V$.
\end{itemize}
In view of~\cite[Fact~3.1]{Sapir-15}, it follows that $\mathbf V$ satisfies the set $\{xzxyty\approx xzyxty,\,xzytxy\approx xzytyx\}$ or its dual $\{xzxyty\approx xzyxty,\,xyzxty\approx yxzxty\}$.
According to~\cite[Proposition~4.1]{Gusev-25IJAC}, $\mathbf L$ and $\mathbf P$ are the only almost Cross varieties of aperiodic monoids satisfying $\{xzxyty\approx xzyxty,\,xzytxy\approx xzytyx\}$.
This fact and~($\ast$) imply that $\mathbf V$ is a Cross variety.

\medskip

\noindent\textsc{Case~2}: $M(xyx)\notin\mathbf V$. 
In view of Lemma~\ref{L: isoterm}, the variety $\mathbf V$ satisfies a non-trivial identity of the form $xyx\approx \mathbf w$. 
Since $xy$ is an isoterm for $\mathbf V$, we have $\mathbf w =x^pyx^q$ for some $p$ and $q$ such that $p\ge 2$ or $q\ge 2$. 
By symmetry, we may assume that $q\ge 2$. 
If $p=0$, then $\mathbf V$ satisfies the identity $x^2\approx x^q$ and, therefore, the identity $xyx\approx yx^2$. 
It is proved in~\cite[the dual to Corollary~3.6]{Lee-14} that every periodic variety that satisfies the last identity is Cross. 
Hence $\mathbf V$ is a Cross variety. 
So, we may further assume that $p\ge 1$. 
Then $\mathbf V$ satisfies the identity $x^2\approx x^{p+q}$. 
Since the variety $\mathbf V$ consists of aperiodic monoids, it satisfies the identity $x^n\approx x^{n+1}$ for some $n\ge1$. 
This identity together with $x^2\approx x^{p+q}$ imply the identity $x^2\approx x^3$.  
It can be easily deduced from this that the identities $xyx\approx x^pyx^q\approx x^pyx^{q+1}\approx xyx^2$ are satisfied by $\mathbf V$.

If $\overleftarrow{\mathbf D} \nsubseteq \mathbf V$, then it follows from~\cite[the dual of Lemma~4.3]{Gusev-25SF} that $\mathbf V$ satisfies $x^2y \approx x^2yx$ and so the identities
\[
xyxz\stackrel{xyx\approx xyx^2}\approx xyx^2z\stackrel{x^2y\approx x^2yx}\approx xyx^2zx\stackrel{xyx\approx xyx^2}\approx xyxzx.
\]
In this case, since $\mathbf V$ does not contain $\mathbf F$ and $\mathbf G$ by~($\ast$), Proposition~\ref{P: xyxz=xyxzx} implies that $\mathbf V$ is a Cross variety.
Hence, we may further assume that $\overleftarrow{\mathbf D}\subseteq \mathbf V$.
Two cases are possible.

\medskip

\noindent\textsc{Case~2.1}:  $\mathbf F\{\alpha_1\}\subseteq \mathbf V$. 
Since $\mathbf V\ne\mathbf I$ by~($\ast$), Lemma~\ref{L: does not contain I} implies that $\mathbf V$ satisfies the identity $xzyxty\approx xzxyxty$. 
Further, since $\mathbf V\ne\mathbf K$  by~($\ast$), it follows from Lemma~\ref{L: does not contain K1} that $\mathbf V$ satisfies the identity $xty^2x \approx xty(yx)^2$.
Consequently, $\mathbf V$ satisfies
\[
xsytyx\stackrel{xyx\approx xyx^2}\approx xsyty^2x \stackrel{xty^2x \approx xty(yx)^2} \approx xsyty(yx)^2\stackrel{xyx\approx xyx^2}\approx xsyt(yx)^2 \stackrel{xtyxsy\approx xtxyxsy}\approx xsytxyx.
\]
Hence $\mathbf V\subseteq\mathbf O$. 
If $Q \notin \mathbf V$, then it follows from Lemma~\ref{L: Q nsubseteq V} that $\mathbf V$ satisfies $x^2yzx \approx x^2yxzx$, whence~$\mathbf V$ is Cross by Lemma~\ref{L: O}.
Hence assume that $Q\in \mathbf V$.
Then, since $\mathbf V\ne\mathbf P$ by ($\ast$), it follows from Lemma~\ref{L: does not contain P} that $\mathbf V$ satisfies $\sigma_k$ for some $k\ge 2$.
Now Lemma~\ref{L: O} applies again, yielding that the $\mathbf V$ is a Cross variety.

\medskip

\noindent\textsc{Case~2.2}: $\mathbf F\{\alpha_1\}\nsubseteq \mathbf V$. 
Then~\cite[Lemma~2.10]{Gusev-25IJAC} implies that $\mathbf V$ satisfies the identity $xyx^2\approx x^2yx$. 
Assume that $A_0\notin\mathbf V$. 
Then $\mathbf V$ satisfies $x^2y^2\approx (xy)^2$ by Lemma~\ref{L: A01 nsubseteq V}.
Since $\mathbf V$ excludes the varieties $\mathbf B$, $\mathbf R$ and $\overleftarrow{\mathbf R}$ by~($\ast$), it follows from Proposition~\ref{P: xyx=xxyx=xyxx,xxyy=xyxy} that $\mathbf V$ is a Cross variety.
So, we may further assume that $A_0\in\mathbf V$. 
Since $\mathbf Z=\mathbf A\vee \overleftarrow{\mathbf A}\nsubseteq \mathbf V$, either $\mathbf A\nsubseteq \mathbf V$ or $\overleftarrow{\mathbf A}\nsubseteq \mathbf V$. 
By symmetry, we may assume without loss of generality that $\overleftarrow{\mathbf A}\nsubseteq \mathbf V$. 
If $\mathbf R_2\nsubseteq\mathbf V$, then $\mathbf V$ satisfies the identity $xty^2x\approx xtxy^2x$ by Lemma~\ref{L: E1{sigma2} or dA1 nsubseteq V}. 
In this case,  since $\mathbf H,\overleftarrow{\mathbf R}\nsubseteq\mathbf V$ by~($\ast$), Proposition~\ref{P: xyx=xxyx,xtyyx=xtxyyx} implies that $\mathbf V$ is a Cross variety.
Now assume that $\mathbf R_2\subseteq\mathbf V$.
Then $\overleftarrow{\mathbf R_2}\nsubseteq\mathbf V$ because $\mathbf V$ excludes $\mathbf S=\mathbf A_0\vee \mathbf R_2\vee \overleftarrow{\mathbf R_2}$, and $\mathbf A\nsubseteq\mathbf V$ because $\mathbf V$ excludes $\mathbf N=\mathbf A\vee \mathbf R_2$.
Now the results dual of Proposition~\ref{P: xyx=xxyx,xtyyx=xtxyyx} and Lemma~\ref{L: E1{sigma2} or dA1 nsubseteq V}  apply, yielding that $\mathbf V$ is a Cross variety.
\end{proof}

\paragraph{\textbf{Acknowledgements.}} The author thanks Edmond W. H. Lee and the anonymous referee for their comments and suggestions for improving the manuscript.

\end{document}